\documentclass[12pt] {article}
\usepackage{amsmath, amsfonts, amsthm, amssymb, color, hyperref, extarrows, enumitem}

\newtheorem{theorem}{Theorem}[section]

\newtheorem{definition}[theorem]{Definition}
\newtheorem{cor}[theorem]{Corollary}

\numberwithin{equation}{section}

\usepackage[T1]{fontenc}

{\theoremstyle{definition}
 
}

\usepackage[titletoc]{appendix}
\newcommand{\norm}[1]{\left\lVert#1\right\rVert}

\usepackage{cite}
 
\hypersetup{
	colorlinks   = true,
	citecolor    = magenta}

	\usepackage[titletoc]{appendix}

 \usepackage{tikz}
 
\usepackage{capt-of}

\begin{document}

\title{\vspace{-1.2cm} \bf Local Rigidity of the Bergman Metric and\\ of the K\"ahler Carath\'eodory Metric}

\author{Robert Xin Dong,\quad Ruoyi Wang,\quad and\quad Bun Wong}
\date{} 
\maketitle

\begin{abstract}
We prove that if the Carath\'eodory metric on a strictly pseudoconvex domain with a smooth boundary is
 locally K\"{a}hler near the boundary,  then the domain is biholomorphic to a ball. 
We also establish a local rigidity theorem for domains with Bergman metrics of constant holomorphic sectional curvature, and highlight this relationship with the Lu constant.

\end{abstract}

\renewcommand{\thefootnote}{\fnsymbol{footnote}}
\footnotetext{\hspace*{-7mm} 
\begin{tabular}{@{}r@{}p{16.5cm}@{}}
& Keywords. Carath\'eodory metric, Bergman metric, K\"{a}hler metric, K\"ahler-Einstein metric, biholomorphic map, holomorphic sectional curvature, Lu constant, strictly pseudoconvex domain, rigidity theorem\\
& Mathematics Subject Classification. Primary 32F45; Secondary 32T15, 32H02 \\

\end{tabular}}

\section{Introduction}

A bounded pseudoconvex domain in complex Euclidean space possesses several biholomorphically invariant metrics: the Bergman metric, the differential Carath\'eodory metric, the unique (up to scaling) K\"ahler-Einstein metric, and the differential Kobayashi-Royden metric. These metrics all coincide, up to a multiplicative constant, on the unit ball $\mathbb{B}^n$. However, this coincidence does not generally occur for arbitrary domains, leading to the following natural question:

     \medskip

\noindent{}{\bf Question.} What are the domains where the Bergman, Carath\'eodory, K\"ahler-Einstein, and Kobayashi metrics are not pairwise distinct (up to scaling)?

   \medskip

For a bounded domain $\Omega$ in $\mathbb{C}^n$, denote by $\mathbb{D}(\Omega)$ the family of holomorphic functions from $\Omega$ into $\mathbb{D}$, and denote by $\Omega(\mathbb{D})$ the family of holomorphic mappings from $\mathbb{D}$ into $\Omega$, where $\mathbb{D}$ denotes the unit disk in $\mathbb{C}$. Recall (see \cite{Wo}) that:

\begin{enumerate}
\item [$1.$] The differential {\it Carath\'eodory metric} is given by
\begin{equation*}
\begin{split}
&C_{\Omega}: \Omega \times \mathbb{C}^n \to \mathbb{R}^+ \cup \{0\},\\
&C_{\Omega}(z, v) = \sup\{ |df(v)| : f \in \mathbb{D}(\Omega), f(z) = 0\},
\end{split}
\end{equation*}
where $df(v)$ is measured with respect to the Poincar\'e metric on $\mathbb{D}$;

\item [$2.$] The differential {\it Kobayashi metric} is given by
\begin{equation*}
\begin{split}
&K_{\Omega}: \Omega \times \mathbb{C}^n \to \mathbb{R}^+ \cup \{0\},\\
&K_{\Omega}(z, v) = \inf\{ |t| : t \text{ is a tangent vector to } \mathbb{D} \text{ at } 0,\\
&\exists f \in \Omega(\mathbb{D}) \text{ with } f(0) = z \text{ and } df(0;t) = v\},
\end{split}
\end{equation*}
where $t$ is measured with respect to the Poincar\'e metric on $\mathbb{D}$.

\end{enumerate}

In the seventies, the third author observed in \cite{W} that on a complex manifold \(M\), the holomorphic sectional curvature of the differential Kobayashi metric \(K_M\) is bounded below by \(-2\), and that of the differential Carath\'eodory metric \(C_M\) is bounded above by \(-2\). With this observation, one can immediately conclude the following characterization of the unit ball.

\begin{theorem}[\cite{W}] 
\label{Theorem (B)}
Let $M$ be a simply-connected complete hyperbolic complex manifold with $C_M = K_M$, assumed to be a K\"ahler metric. Then, $M$ is biholomorphic to a ball.
\end{theorem}

A \emph{K\"ahler metric} is a Hermitian metric $g = g_{\alpha \bar \beta}\, dz^\alpha \otimes d\bar z^\beta$ on a complex manifold whose associated fundamental $(1,1)$-form $\omega = \sqrt{-1}\, g_{\alpha \bar \beta}\, dz^\alpha \wedge d\bar z^\beta$ is closed, i.e., $d\omega = 0$. Equivalently, $g$ is K\"ahler if and only if it locally admits a K\"ahler potential: $g_{\alpha \bar \beta} = \frac{\partial^2 \varphi}{\partial z_\alpha \partial \bar z_\beta}$ for some smooth real-valued function $\varphi$.  A K\"ahler manifold is a complex manifold equipped with a K\"ahler metric. Stanton proved the following result, which is more general than Theorem \ref{Theorem (B)}.

\begin{theorem}[\cite{Sta}] \label{Stanton}
Let $M$ be a complete hyperbolic complex manifold. Assume that there is a point on $M$ at which $C_M$ and $K_M$ are equal, and one of these metrics is Hermitian and of class $C^\infty$. Then $M$ is biholomorphic to a ball.
\end{theorem}

Gaussier and Zimmer recently established the following beautiful
relationship between local curvature conditions and global structure for complete K\"ahler metrics (see Corollary 1.2 and its proof in \cite{GZ}, and a previous result of Seshadri and Verma \cite{SeV} for simply connected non-positively curved K\"ahler manifolds), which   shows that constant negative holomorphic sectional curvature on the complement of a compact set ensures the universal cover is biholomorphic to a ball.

\begin{theorem}[\cite{GZ}] \label{universal_cover_is_a_ball}
Suppose that $M$ is a Stein manifold with $\dim_{\mathbb{C}} M \geq 2$ and $g$ is a complete K\"ahler metric on $M$. If there exists a compact set $\Sigma \subset M$ such that $M \backslash \Sigma$ is connected and $g$ has constant negative holomorphic sectional curvature on $M \backslash \Sigma$, then
  the universal cover of $M$ is biholomorphic to a ball in $\mathbb{C}^{\dim_{\mathbb{C}} M}$.
\end{theorem}

 A famous theorem of Lu \cite{Lu} says that {\it a bounded domain in $\mathbb{C}^n$ with a complete Bergman metric of constant holomorphic sectional curvature must be biholomorphic to $\mathbb{B}^n$}. In the first main result of this paper, we give an alternative argument for, and a slightly improved version of, Lu's theorem:

\begin{theorem} \label{nice cor}
Let $\Omega$ be a bounded domain in $\mathbb{C}^n$. Suppose the Bergman metric $B^{2}_{\Omega}$ is complete and its holomorphic sectional curvature is identically equal to a negative constant $-c^2$ on an open set $U \subset \Omega$. Then, $\Omega$ is biholomorphic to $\mathbb{B}^n$, where $n = 2c^{-2}-1$.
\end{theorem}

While the global constant curvature follows from the identity theorem of real-analytic functions, our proof provides an explicit construction of the Bergman representative coordinate $T$ as a holomorphic isometry, which has independent interest and applications. Similar uniqueness principles based on real-analyticity appear in our analysis of strictly pseudoconvex domains. In \cite{DWo, DWo2}, the first and third authors extended Lu's theorem to the Bergman-incomplete situation and provided a multidimensional generalization of Carath\'eodory's theorem on the continuous extension of biholomorphisms up to the closures. See also the recent works \cite{HL} by Huang and Li, and \cite{ETX} by Ebenfelt, Treuer and Xiao.

   \medskip

These geometric insights naturally lead us to examine the quantitative relationships between different metrics on bounded domains. A foundational result in this direction is Lu's inequality, which establishes precise control of the Carath\'eodory metric by the Bergman metric:

\begin{theorem}[\cite{Look}, cf. \cite{CXZ}] \label{Lu's}
Let $\Omega$ be a bounded domain in $\mathbb{C}^n$. For each tangent vector $v \in T_z(\Omega) = \mathbb{C}^n$ at $z \in \Omega$, it holds that
$$
B_{\Omega}(z, v) \geq C_{\Omega}(z, v),
$$
where $B_{\Omega}(z, v)$ and $C_{\Omega}(z, v)$ are the lengths of $v$ with respect to the Bergman metric $B^2_{\Omega}$ and the differential Carath\'eodory metric $C_{\Omega}$, respectively.
\end{theorem}

The above inequality due to Lu, while fundamental, is not sharp. To obtain optimal control of the relationship between these metrics, we consider the following Lu constant.

     \medskip

\noindent{}{\bf Definition.} The {\it Lu constant} $L(\Omega)$ of a bounded domain $\Omega$ in $\mathbb{C}^n$ is defined as
$$
L(\Omega) := \sup_{z \in \Omega; \, 0 \neq v \in T_z(\Omega)} \frac{C_{\Omega}(z, v)}{B_{\Omega}(z, v)}.
$$

\medskip

The definition of $L(\Omega)$ implies that
\begin{equation} \label{sharp}
L(\Omega) B_\Omega \geq C_\Omega.
\end{equation}
For the rest of this paper, we shall call \eqref{sharp} the  {\it sharp Lu's inequality}. The value of the Lu constant $L(\Omega)$ is understood in the following scenarios:

\medskip
 
By Theorem \ref{Lu's}, \(L(\Omega) \leq 1\). 
Also, \(L(\Omega) = \frac{1}{\sqrt{n+1}}\) when \(\Omega\) is the unit ball, and \(L(\Omega) = \frac{1}{\sqrt{2}}\) when \(\Omega\) is the unit polydisc \(\mathbb{D}^n\) (since $C_{\mathbb{D}^n}(0,v) = \max_j |v_j|$ and $B_{\mathbb{D}^n}(0,v) = \sqrt{2}\|v\|$ in the Poincar\'e metric with holomorphic sectional curvature $-2$, and the supremum of $C_{\mathbb{D}^n}/B_{\mathbb{D}^n}$ is attained in any coordinate direction). One further determines the Lu constants of all Cartan classical domains.
When $\Omega$ is a bounded strictly pseudoconvex domain in $\mathbb{C}^n$, $L(\Omega) \geq \frac{1}{\sqrt{n+1}}$ by the boundary asymptotic behaviors of $B_{\Omega}$ and $C_{\Omega}$ due to Diederich \cite{Di} and Graham \cite{G}, respectively, as $C_\Omega/B_\Omega$ approaches $\frac{1}{\sqrt{n+1}}$ near the boundary.
 If the holomorphic sectional curvature of the Bergman metric on a bounded strictly pseudoconvex domain in $\mathbb{C}^n$ is bounded below by $-\frac{2}{n+1}$, then \(L(\Omega) = \frac{1}{\sqrt{n+1}}\) following from the Ahlfors-Schwarz lemma.

 \medskip

Next, along the line of the Wong/Stanton theorem, we present a rigidity result on the comparison between the Bergman metric and  the Carath\'eodory metric.

    \medskip

\begin{theorem} \label{rigidity}

\item [{$\bf A$.}] Let $X$ be a hyperbolic Riemann surface.
Assume that the Bergman metric $B(z)|dz|^2$ is complete with holomorphic sectional curvature no less than \(-1\), 
and 
the Carath\'eodory metric is written as $c^2_B(z) |dz|^2$,
where $z$ is a local coordinate.
Then, it holds on $X$ that
 $$
{B} \geq   2c^2_B.
 $$
Moreover, $X$ is biholomorphic to a disc if and only if there exists some $z_0 \in X$ such that
$$
B(z_0) =   2 c^2_B(z_0).
$$

 \item [{$\bf B$.}]

Let $\Omega \subset \mathbb{C}^n$ be a  domain whose
  universal cover  is biholomorphic to a Euclidean ball. Assume that  the   Bergman metric $B^2_{\Omega}$  is complete on $\Omega$ with holomorphic sectional curvature no less than $-2(L(\Omega))^2$. Then, $\Omega$ is biholomorphic to a ball  if and only if   there exists some  $z_0 \in \Omega$ such that 
$$
 L(\Omega) B_\Omega (z_0, v) = C_{\Omega} (z_0, v), \quad \text{for all\,} \, v\in T_{z_0}(\Omega). 
$$

\end{theorem}

\medskip

Since a bounded pseudoconvex domain with a spherical boundary in \(\mathbb{C}^n\), \(n \geq 2\), is universally covered by a ball by the uniformization theorem (see \cite[Theorem A.2]{NS}), we obtain the following corollary   relating to the examples of Burns and Shnider \cite{BS1976} using the aforementioned methods:

\medskip

\begin{cor} \label{cor}

{\it Let $\Omega$ be a bounded strictly pseudoconvex domain  with spherical boundary in $\mathbb{C}^n$, $n \geq 2$. Then 
$\Omega$ is not biholomorphic to a Euclidean ball if and only if either the 
holomorphic sectional curvature of the Bergman metric   in some direction is strictly less than 
$-\frac{2}{n+1}$ or the sharp Lu's inequality \eqref{sharp} is strict at all points and for some directions.}

\end{cor}

In the second half of the paper, we turn our attention to bounded domains where the Carath\'eodory metric exhibits special geometric properties - specifically, when it coincides up to scaling with the Bergman or K\"ahler-Einstein metric, or more generally, when it is a K\"ahler metric. This investigation connects to a broader question raised (in a slightly different form)  by Yau \cite[pp. 679]{Yau} concerning the characterization of manifolds whose Bergman metrics are K\"{a}hler-Einstein.

\medskip

On strictly pseudoconvex domains, significant progress has been made in understanding the relationships between these metrics. Cheng and Yau \cite{CY} showed the existence of complete K\"ahler-Einstein metrics on smoothly bounded strictly pseudoconvex domains  in $\mathbb{C}^n$, while Mok and Yau \cite{MY83} later extended this to general bounded pseudoconvex domains without boundary regularity conditions. 
For smoothly bounded strictly pseudoconvex domains, Cheng conjectured that the Bergman metric is K\"{a}hler-Einstein if and only if the domain is biholomorphic to a ball. Following the earlier works of Fu and the third author \cite{FW} and Nemirovski and Shafikov \cite{NS}, this conjecture was confirmed by Huang and Xiao in \cite{HX}. Additionally, see the recent work \cite{SX} by Savale and Xiao for finite type domains in $\mathbb{C}^2$.

     \medskip

  Building on these foundations and inspired by recent work of Gaussier and Zimmer  \cite{GZ}, who characterized domains whose universal coverings are biholomorphic to a ball through properties of the Kobayashi metric, we establish parallel results for the Carath\'eodory metric. Their characterization showed:

\begin{theorem}[\cite{GZ}] \label{thm:spcv}
Suppose that $\Omega \subset \mathbb{C}^n$, $n \geq 2$, is a bounded strictly pseudoconvex domain with a $C^2$ boundary. Then, the following are equivalent:
\begin{enumerate}
\item [$(a)$] The Kobayashi metric on $\Omega$ is a K\"ahler metric;
\item [$(b)$] The Kobayashi metric on $\Omega$ is a K\"ahler metric with constant holomorphic sectional curvature;
\item [$(c)$] The universal cover of $\Omega$ is biholomorphic to a ball;
\item [$(d)$] The Kobayashi metric on $\Omega$ is a K\"ahler-Einstein metric.
\end{enumerate}
Moreover, the Kobayashi metric on $\Omega$ is a scalar multiple of the Bergman metric if and only if $\Omega$ is biholomorphic to a ball.
\end{theorem}

 For strictly pseudoconvex domains, we obtain the following main result, which is a Carath\'eodory-analogue of both Gaussier-Zimmer's theorem and the resolution of the Cheng conjecture:

\begin{theorem} \label{Main}
Suppose that $\Omega \subset \mathbb{C}^n$, $n \geq 2$, is a bounded strictly pseudoconvex domain with a $C^2$ boundary. Then, the following are equivalent:
\begin{enumerate}
\item [$(1)$] The Carath\'eodory metric on $\Omega$ is a K\"ahler metric on $U \cap \Omega$, where $U$ is an open set that contains the boundary of $\Omega$;
\item [$(2)$] The Carath\'eodory metric on $\Omega$ is a K\"ahler metric with constant holomorphic sectional curvature;
\item [$(3)$] $\Omega$ is biholomorphic to the unit ball;
\item [$(4)$] The Carath\'eodory metric on $\Omega$ is a scalar multiple of the Bergman metric;
\item [$(5)$] The Carath\'eodory metric on $\Omega$ is a K\"ahler-Einstein metric.
\end{enumerate}
\end{theorem}

     \medskip
     
For $n=1$, Suita \cite{Su73} (cf. \cite{Burbea}) showed that the
Carath\'eodory metric is real analytic for any bounded domain in $\mathbb{C}$.  
Thus, it can be easily proven that if the Carath\'eodory metric is 
 complete with holomorphic sectional curvature identically equal to $-2$, then the domain 
is biholomorphic to a disc.

     \medskip

Some related examples are discussed below. These examples illustrate the boundary between K\"ahler and non-K\"ahler behavior of the Carath\'eodory metric across several important classes of domains. The question of when the Carath\'eodory metric is K\"ahler is subtle and the answer is non-obvious; the following collection serves as a reference in this area.

\bigskip

\noindent{}{\bf Examples.}

\begin{enumerate} 

\item [$(i)$] By the remarkable works of Lempert \cite{Lem, Lem81, Lem88}, the Carath\'eodory metric and the Kobayashi metric coincide on convex domains. Theorem \ref{Theorem (B)}/\ref{Stanton} implies that the Carath\'eodory metric on a bounded convex domain is a Hermitian metric if and only if the domain is biholomorphic to the unit ball. In particular, let $\mathcal{P}$ be a  small generic perturbation of $\mathbb{B}^n$ such that $\mathcal{P}$ is convex in $\mathbb{C}^n$ but not biholomorphic to $\mathbb{B}^n$ when $n \geq 2$ (see the papers of Burns, Shnider and Wells \cite{BSW1978}, and Greene and Krantz\cite{GK}). Then, the Carath\'eodory metric on $\mathcal{P}$ is \textit{not} a Hermitian metric when $n \geq 2$.

In \cite{DT}, the first author and Treuer constructed a bounded complete Reinhardt domain 
$$
\mathcal{R} := \{z \in \mathbb{C}^2 : |z_1|^4 + |z_1|^2 + |z_2|^2 < 1\}
$$
which is strongly convex with an algebraic boundary and not biholomorphic to $\mathbb{B}^2$. By the above discussion, the Carath\'eodory metric on $\mathcal{R}$ is \textit{not} a Hermitian metric.

Also, note that a complex analytic ellipsoid 
 $$
 {E}:=\{(z,w)\in {\mathbb{C}^1}\times{\mathbb{C}^1} : |z|^2+|w|^{2p}<1 \},$$
is convex and
not biholomorphic to $\mathbb{B}^2$,  where $\mathbb Z^+ \ni p>1$.
So,  the Carath\'eodory  metric on ${E}$
is  \textit{not} a Hermitian metric. 
(See some recent related work in this direction by Cho \cite{Cho}).

    \item [$(ii)$] The Kobayashi metric of the unit bidisc is 
$$
K_{\mathbb{D}^2} ((z, w), (X, Y)) = \max \{ K_{\mathbb{D}}(z, X), K_{\mathbb{D}}(w, Y) \}.
$$
Thus, the Carath\'eodory metric and the Kobayashi metric are not $C^1$-smooth metrics in general. For bounded symmetric domains of rank $\geq 2$, the Carath\'eodory and Kobayashi metrics are \textit{not} K\"ahler.

\item [$(iii)$] Burns and Shnider \cite{BS1976} have constructed examples of non-simply connected strictly pseudoconvex domains with spherical boundaries in $\mathbb{C}^n$. By the uniformization theorem (see \cite{NS1, NS2}), such domains are universally covered by the unit ball, so the Kobayashi metric is K\"ahler-Einstein. For example, let
$$
\mathcal{S} := \{(z, w) \in \mathbb{C}^2 \mid \sin \log|z| + |w|^2 < 0, e^{-\pi} < |z| < 1\}.
$$
By our Theorem \ref{Main} above, the Carath\'eodory metric $C_{\mathcal{S}}$ is \textit{not} a K\"ahler metric on $\mathcal{S}$, and thus $C_{\mathcal{S}}$ is not equal to the Kobayashi metric $K_{\mathcal{S}}$.
\end{enumerate}

 \medskip

\medskip

Lastly, we extend a previous result \cite{CW}  (see also \cite{CW00}) by Cheung and the third author from bounded convex domains  to 
strictly pseudoconvex domains.

\medskip

\begin{theorem} \label{2nd}
Let $\Omega \subset \mathbb{C}^n$, $n \geq 2$, be a bounded strictly pseudoconvex domain with a $C^2$ boundary and the Lu constant $L(\Omega)$. Suppose the holomorphic sectional curvature of the Bergman metric $B^2_{\Omega}$ on $\Omega$ is bounded above by $-2(L(\Omega))^2$. Then $\Omega$ is biholomorphic to a ball.
\end{theorem}

      \medskip

Our proofs of Theorems \ref{Main} and \ref{2nd} depend heavily on  the work of Bracci, Forn{\ae}ss, and Wold \cite{BFW19} and that of Gaussier and Zimmer  \cite{GZ}. (For related observations, see earlier works by Lempert \cite{Lem84} and Huang \cite{H94}.) One of the key ideas is that at a point sufficiently close to the boundary of a smoothly bounded strictly pseudoconvex domain, there is a nonempty, open set of directions where the Kobayashi and Carath\'eodory metrics coincide. We observe that the Carath\'eodory metric has constant negative holomorphic sectional curvature in these directions, and in fact, has constant negative holomorphic sectional curvature near the boundary, if the metric is 
K\"ahler. The proof of Theorem  \ref{2nd} also relies on our  Theorem \ref{nice cor}.

\medskip

The organization of the paper is as follows. In Section 2, we give a proof of our main result (Theorem \ref{Main}) characterizing when the Carath\'eodory metric is K\"ahler. 
In Section 3, we establish our 
local-to-global principle for the Bergman metric (Theorem \ref{nice}) which yields our first main theorem (Theorem \ref{nice cor}), and then provide proofs of our rigidity results (Theorems \ref{rigidity} and \ref{2nd}) relating the Lu constant to geometric properties of domains.

\section{When is the Carath\'eodory metric a K\"ahler metric?}

We denote by $\| \cdot \|$ the Euclidean norm in $\mathbb{C}^n$ and by $\mathbb{B}^n$ the unit ball $\mathbb{B}^n := \{z \in \mathbb{C}^n : \|z\| < 1\}$ in $\mathbb{C}^n$. When $n = 1$, we let $\mathbb{D}$ denote the unit disk in $\mathbb{C}$. If $\Omega \subset \mathbb{C}^n$ is a bounded domain and $z \in \Omega$, denote by $\delta_\Omega(z)$ the distance function to the boundary:
$$
\delta_\Omega(z) = \min\{ \|z - x\| : x \in \partial \Omega \}.
$$
Throughout the paper, we use two equivalent descriptions for subsets of $\Omega$ near the boundary: $U \cap \Omega$ where $U$ is an open set containing $\partial \Omega$, and $\Omega \setminus \Sigma$ where $\Sigma$ is a compact subset of $\Omega$. These are interchangeable, as $\Sigma = \{z \in \Omega : \delta_\Omega(z) \geq \epsilon\}$ for sufficiently small $\epsilon > 0$.

\begin{definition}

  A holomorphic map $\varphi : \mathbb{D} \rightarrow \Omega$ is a {\it complex geodesic} with respect to the Kobayashi (or Carath\'eodory) metric if
$$
k_\Omega(\varphi(z_1), \varphi(z_2)) = k_{\mathbb{D}}(z_1, z_2)
\quad (\text{or \,} c_\Omega(\varphi(z_1), \varphi(z_2)) = c_{\mathbb{D}}(z_1, z_2))
$$
for all $z_1, z_2 \in \mathbb{D}$, where $k_\Omega$ (or $c_\Omega$)  denotes the Kobayashi (or Carath\'eodory) distance.

\end{definition}

\medskip

 Suppose that 
$ {C_{\Omega}(z, v)} = {K_{\Omega}(z, v)}$ for $z \in \Omega, \,  v\in T_z(\Omega)$, where $\Omega$ is a bounded domain  in $\mathbb{C}^n$.
By the definitions of the Carath\'eodory and
Kobayashi  metrics, let 
$\rho: \Omega \to \mathbb{D}$ be a holomorphic function that realizes  the Carath\'eodory metric $C_\Omega (z, v)$,
and 
let $\varphi: \mathbb{D} \to \Omega$ be a holomorphic map that realizes the Kobayashi metric $K_\Omega (z, v)$.
Then, the composition $\rho \circ \varphi$ is an automorphism of $\mathbb D$ that fixes the origin.
It can be shown that $\varphi: \mathbb{D} \to \Omega$ is a complex geodesic with respect to both the Carath\'eodory and Kobayashi metrics. Furthermore, the holomorphic sectional curvatures of $C_{\Omega}$ and $K_{\Omega}$ restricted to the image $\varphi(\mathbb{D})$ are both equal to $-2$ everywhere.

\medskip

When $\Omega$ has a $C^2$ boundary and $z \in \overline{\Omega}$ is sufficiently close to $\partial \Omega$, there is a unique point $\pi(z) \in \partial \Omega$ such that $\|z - \pi(z)\| = \delta_\Omega(z)$. We let
$$
P_z : \mathbb{C}^n \rightarrow T_{\pi(z)}^{\mathbb{C}} \partial \Omega
$$
denote the (Euclidean) orthogonal projection and let $P_z^{\bot} = \text{Id} - P_z$. The following result states that complex geodesics exist at base points sufficiently close to the smooth boundary of strictly pseudoconvex domains and in directions sufficiently tangential.

\medskip

\begin{theorem}[\cite{BFW19, GZ}] \label{local}
Suppose that $\Omega \subset \mathbb C^n, n \ge 2,$ is a bounded strictly pseudoconvex domain with $C^2$ boundary. 
Then 
 there exists $\epsilon > 0$ such that: if $z \in \Omega$ and $\delta_\Omega(z) < \epsilon$, then there exists an non-empty open set $E_\Omega(z)$ of non-zero tangent vectors at $z$ where the Kobayashi and Carath\'eodory metrics agree.

   \end{theorem}

\medskip

Theorem \ref{local} was first established for $C^3$ domains by Bracci, Forn{\ae}ss, and Wold \cite{BFW19}. The $C^2$ case stated here is due to Gaussier and Zimmer \cite{GZ}.

\medskip

Theorem \ref{local} is fundamental for our analysis. The underlying ideas have a rich history. The existence of complex geodesics in strictly pseudoconvex domains was first established by Lempert \cite{Lem84},  
who showed that the relevant estimates depend only on the normal curvature of the domain.  Lempert's original analysis demonstrated that since all estimates depend only on normal curvature, the results naturally extend to the $C^2$ case through Montel's theorem, even though his original proof was given for the analytic case.
Moreover, Huang \cite{H94} proved that certain extremal mappings near boundary points are complex geodesics (denoted here as $\varphi$). See also \cite{BK1994, Ber} for related results. 
The existence of an open set where the Kobayashi and Carath\'eodory metrics coincide was first shown by Kosi{\'n}ski \cite{Ko}. 
In \cite[Theorem 1.1]{BFW19},  the precise conditions 
for when the Kobayashi and Carath\'eodory metrics coincide near the boundary were established by showing
 $$
 {C_{\Omega}(z, v)} = {K_{\Omega}(z, v)}, \quad z \in \Omega, \,  v\in T_z(\Omega)       
 $$
 if $\delta_\Omega(z) < \epsilon$ 
 and if
 $$
\norm{P_z^{\bot}(v)} < \epsilon \norm{P_z(v)}.
$$  
Gaussier and Zimmer \cite{GZ}   also obtained the following result which will play a crucial role in our arguments.

\medskip

\begin{theorem}[\cite{GZ}] \label{locally symmetric}
Suppose that $M$ is a Stein manifold with $\dim_{\mathbb{C}} M \geq 2$,
$\Sigma \subset M$ is a compact set where $M \backslash \Sigma$ is connected, and $g_0$ is a Hermitian metric on $M \backslash \Sigma$ which is complete at infinity.  If $g_0$ is locally symmetric metric, then there
exists a  complete locally symmetric Hermitian metric $g$ on $M$ such that $g = g_0$ on $M \backslash \Sigma$.
 \end{theorem}

\medskip

Theorem \ref{locally symmetric} implies Theorem \ref{universal_cover_is_a_ball} since a K\"ahler metric with constant holomorphic sectional curvature is locally symmetric.

\medskip

  \medskip

\begin{proof}  [{\bf Proof of Theorem \ref{Main}}]
Since $(2)$ implies $(1)$ by definition, and it is well known that $(3)$ implies $(2)$, we only need to prove the non-trivial part, namely that $(1)$ implies $(3)$. By Theorem \ref{local}, there exists some $\epsilon > 0$ such that: if $z \in \Omega$ and $\delta_\Omega(z) < \epsilon$, then there is a nonempty, open set $E_\Omega(z)$ of directions where the Carath\'eodory metric and the Kobayashi metric coincide.
Using the argument in \cite{W}, we know that 
the holomorphic sectional curvature $H(z; v)$ of the Carath\'eodory metric is equal to $-2$
 for any non-zero $v \in E_\Omega(z)$.
Briefly speaking, given that the Carath\'eodory metric and Kobayashi metric coincide at \(z\) for \(v\), there exists a complex geodesic with respect to the Carath\'eodory metric on which the holomorphic sectional curvature is equal to \(-2\) at every point. As the curvature \( H(z; v) \) of \(\Omega\) is bounded above by \(-2\), the curvature-decreasing property implies that it must be exactly \(-2\) at \(z\).  
Moreover, since the holomorphic sectional curvature function
\[
v \in T_z(\Omega) \mapsto H(z; v) \in \mathbb{R}
\]
is algebraic (with \(z\) fixed), 
it follows that the holomorphic sectional curvature \( H(z; v) \) equals \(-2\) for all non-zero vectors \( v \in T_z(\Omega) \).
This is true for all  $z\in \Omega$ such that 
$ \delta_\Omega(z) < \epsilon$ for a sufficiently small $\epsilon$.
Define the compact set 
\begin{equation} \label{sigma}
\Sigma := \{z \in \Omega : \delta_\Omega(z) \ge \epsilon\}
\end{equation}
as in Theorem~\ref{universal_cover_is_a_ball}. Then, the holomorphic sectional curvature of the Carath\'eodory metric $C_\Omega$ is equal to $-2$ on $\Omega \setminus \Sigma$.

\medskip

By the work of Graham \cite[Proposition 5]{G}, the Carath\'eodory metric $C_\Omega$ is complete on $\Omega$. 
If $C_\Omega$ is a K\"ahler metric on $U \cap \Omega$,    then by 
Theorem \ref{locally symmetric},
 there exists a complete 
  locally symmetric metric ${g}$ on $\Omega$ such that 
\begin{equation} \label{extend}
{g} = C_\Omega \quad \text{on} \quad U \cap \Omega.
\end{equation}
Since the universal cover of $(\Omega, {g})$ is globally symmetric (hence homogeneous) and ${g}$ has constant negative holomorphic sectional curvature on $U \cap \Omega$, we then see that ${g}$ has constant negative holomorphic sectional curvature on $\Omega$. 
  Thus,   the universal cover of $\Omega$ is biholomorphic to the unit ball.
By Theorem \ref{thm:spcv}, the Kobayashi metric $K_\Omega$ is  a complete K\"ahler metric of constant negative holomorphic sectional curvature on $\Omega$. Let $\pi: \mathbb{B}^n \rightarrow \Omega$ be the holomorphic covering map, which is a local isometry for $K_\Omega$.
Hence, we have $\pi^*K_{\Omega} = \pi^* {g}$ with holomorphic sectional curvature equaling $-2$ on $\mathbb{B}^n$. At any point $z\in \Omega \setminus \Sigma$,  it holds by \eqref{extend} that
$$ 
 {C_{\Omega}(z, v)} = { {g}(z, v)} = {K_{\Omega}(z, v)}, \quad  v\in T_z(\Omega).
$$
By Theorem \ref{Stanton}, one concludes that $\Omega$ is biholomorphic to a ball.

\medskip

Regarding Conditions $(4)$ and $(5)$, notice that Condition $(3)$ implies both of them. And either Condition $(4)$ or Condition $(5)$ would imply Condition $(1)$ by definition. Thus, the proof is complete.

\end{proof}

\medskip
 
From the proof of Theorem \ref{Main}, we see that Condition $(a)$ in Theorem \ref{thm:spcv} can be relaxed to:

\medskip

$(a^{\prime})$ {\it The Kobayashi metric on $\Omega$ is a K\"ahler metric on $U \cap \Omega$, where $U$ is an open set that contains the boundary of $\Omega$}.\\

\bigskip

\noindent{}{\bf Remark.} {\normalsize On April 19th, 2024, the first author gave an online talk at Syracuse University's Analysis Seminar hosted by Yuan Yuan.  
During the talk, Yuan mentioned that he had obtained some related  result yet to be written up.
The results in this paper were obtained around  2022/2023.}\\

\section{Local-to-global principle for the Bergman metric}

For a bounded domain $\Omega \subset \mathbb{C}^n$, its  {\it Bergman kernel} is defined as
$$
\mathcal{K}(z, t) := \sum \varphi_j(z) \overline{\varphi_j(t)}, \quad z, t \in \Omega,
$$
where $\{\varphi_j\}_{j=1}^{\infty}$ is a complete orthonormal basis for the space of $L^2$ holomorphic functions. The  {\it Bergman metric} $B^2_\Omega(z, v)$ is defined by
$$
B^{2}_{\Omega}(z, v) := \sum_{\alpha, \beta = 1}^n g_{\alpha \bar{\beta}} v_\alpha \overline{v_\beta}, \quad \text{where } g_{\alpha \bar{\beta}}(z) = \frac{\partial^2 \log \mathcal{K}(z, z)}{\partial z_\alpha \partial \overline{z_\beta}}, \quad \text{for } z \in \Omega,\ v \in \mathbb{C}^n.
$$
The Bergman metric is always a K\"ahler metric and it is incomplete in many cases. One can also define the Bergman kernel and metric on general complex manifolds by considering $L^2$ holomorphic top-forms instead of functions.
The {\it Bergman representative coordinate} $T (z) =(w_1, ... , w_n)^{\tau}$ relative to $p\in  \Omega$ is defined as
\begin{equation} \label{rep}
w_{\alpha} (z):=\sum _{j=1}^{n} g^{\bar j \alpha }(p) \left(\mathcal{K}(z, p) ^{-1} \left. \frac{\partial}{\partial \overline {t_j} } \right|_{t=p} \mathcal{K}(z, t)  -    \left. \frac{\partial}{\partial \overline {t_j} } \right|_{t=p}  \log \mathcal{K}(t, t)\right),
\end{equation}
where $(g^{\bar j \alpha })  = (g_{\alpha \bar{j}})^{-1}$. It is well known  that $T (z)$ is holomorphic on $\Omega $ less the zero set of $\mathcal K(\cdot , p)$, cf. \cite{GKK}.

\medskip

  The Carath\'eodory metric corresponds to the analytic capacity for an open Riemann surface (see \cite{Su73}).

\begin{definition}

Let $X$ be an open hyperbolic Riemann surface.   
The analytic capacity of a Riemann surface $X$ is defined as
 $$
c_B(z_0) = \sup\left \{ \left|{\partial f \over \partial w}(z_0)\right|: f \in \hbox{Hol}(X, \mathbb{D}), f(z_0) = 0\right \}.
$$
\end{definition}

  Ahlfors introduced the analytic capacity for domains in order to study Painlev\'{e}'s question: \textit{which compact sets $E$ in the complex plane are removable for the bounded holomorphic functions}?

\medskip

\begin{proof} [{\bf Proof of Theorem \ref{rigidity}, Part {$\bf A$}}]

Since $\Omega$ is universally
 covered by a disc,
in a local coordinate at $z \in X$, let
$
\lambda^2(z) |dz|^2
$
denote the Poincar\'e metric with constant  holomorphic sectional  curvature \(-2\).
Due to the  Ahlfors-Schwarz lemma  (see \cite{Ahlfors, Roy, Yau78}) and the fact that the Kobayashi metric dominates the Carath\'eodory metric, we get 
$$ 
\frac{B}{2} \geq \lambda^2  \geq c^2_B,
$$
which yields the inequality part.

\medskip

If there exists some $z_0 \in X$ such that 
$$
B(z_0) =  2 c^2_B(z_0),
$$
then $\lambda (z_0) = c_B(z_0)$.
Since the Kobayashi metric and the Carath\'eodory metric are equal at one point,
by a lemma of the third author in  \cite{Wo}, we conclude that $X$ is biholomorphic to a disc.
  The converse is straightforward as both metrics are biholomorphically invariant with explicit formulas on a disc.

\end{proof}

\medskip

We also recall the definition of  the logarithmic   capacity, in order to  state  Theorem \ref{one dim improved}, which is a strengthened version of Theorem \ref{rigidity}, Part {$\bf A$}.
Let $SH^{-}(X)$ denote the set of negative subharmonic functions on an open Riemann surface $X$. Given   $z_0 \in X$,  let $w$ be  a fixed local coordinate  in a neighbourhood of $z_0$ such that $w(z_0) =0$. The (negative) Green's function   is
 $$
G(z, z_0) = \sup\{u(z): u \in SH^{-}(X), \limsup_{z \to z_0} u(z) - \log|w(z)| < \infty\}.
 $$
 An open Riemann surface admits a Green's function if and only if there is a non-constant, negative subharmonic function defined on it.  The Green's function is strictly negative on the surface and harmonic except on the diagonal. The Green's function  on a Riemann surface induces the logarithmic capacity, which is used prominently in potential theory. 

\begin{definition}

Let $X$ be an open hyperbolic Riemann surface.  The logarithmic capacity $c_{\beta}$ is defined as
 $$
c_{\beta}(z_0) = \lim_{z \to z_0} \exp(G(z, z_0) - \log|w(z)|).
 $$

\end{definition}

  When we wish to emphasize the surface $X$,  
we will use the notation $c_{m; X}(\cdot)$, $m = \beta  \text{ or } B$.  If $h:X_1 \to X_2$ is a biholomorphism, then
 $$
c_{m; X_1}(z) = |h'(z)|c_{m; X_2}(h(z)), \quad m = \beta  \text{ or } B.
 $$
Thus, both the logarithmic and analytic capacities are independent of the choice of the local coordinates and define conformally-invariant metrics $c_\beta(z)|dz|$ and $c_B(z)|dz|$.

\medskip

 \begin{theorem} \label{one dim improved}
Let $X$ be a hyperbolic Riemann surface whose Bergman metric $B(z)|dz|^2$ is complete with  holomorphic sectional curvature   no less than \(-1\). In a local coordinate at $z \in X$, let
$$
\lambda^2(z) |dz|^2, \quad c^2_\beta(z) |dz|^2, \quad \text{and} \quad c^2_B(z) |dz|^2
$$
denote the Poincar\'e metric with constant  holomorphic sectional curvature \(-2\), the logarithmic capacity, and the analytic capacity, respectively. Then, it holds on $X$ that
\begin{equation} \label{chain}
\frac{B}{2} \geq \lambda^2 \geq c^2_{\beta} \geq c^2_B;
\end{equation}
moreover, $X$ is biholomorphic to a disc if and only if there exists some $z_0 \in X$ such that
$$
B(z_0) = 2 c^2_\beta(z_0) \quad \text{or} \quad 2 c^2_B(z_0).
$$
\end{theorem}

   \medskip

\begin{proof}

The first inequality in \eqref{chain} holds due to the Ahlfors-Schwarz lemma. The second inequality is a result of Minda, who also characterized its equality part in \cite{Min},
by showing that  the Poincar\'e metric with constant  holomorphic sectional curvature  \(-2\)  coincides with the logarithmic capacity at one point  if and only if the surface is simply connected.
The last inequality is well known and the characterization of its equality part can be deduced from a result of Minda \cite{Min} who used the sharp form of the Lindel\"of principle (see also a more direct proof of the equality characterization without adopting this principle in \cite{DTZ} and some related observations in \cite{D}).

\medskip

If there exists some $z_0 \in X$ such that 
$$
B(z_0) = 2 c^2_\beta(z_0) \quad \text{or} \quad 2 c^2_B(z_0),
$$
then $\lambda (z_0) = c_\beta(z_0)$, and we conclude by \cite{Min} that $X$ is simply connected. The converse is straightforward as all the inequalities in \eqref{chain} become equalities.

\end{proof}

        \medskip

Next, we prove Theorem \ref{rigidity}, Part {$\bf B$} and Corollary \ref{cor}, which involve the Lu constant $L(\Omega)$.

        \medskip

\begin{proof} [{\bf Proof of Theorem \ref{rigidity}, Part {$\bf B$}}]

Since $\Omega$ is holomorphically covered by a ball, the Kobayashi metric   on $\Omega$ is a K\"ahler metric with constant holomorphic sectional curvature.
Due to the  Ahlfors-Schwarz lemma  (see \cite{Ahlfors, Roy, Yau78,Lu1979}) and 
the lower bound of the holomorphic sectional curvature of the complete Bergman metric, we get 
$$
L(\Omega) B_\Omega \geq K_\Omega \geq C_\Omega.
$$

\medskip

If there exists some $z_0 \in \Omega$ such that 
$$
L(\Omega) B_\Omega(z_0, v) =  C_\Omega(z_0, v),  \quad \text{for all\,} \, v\in T_{z_0}(\Omega),
$$
then $K_\Omega (z_0, v) =C_\Omega(z_0, v)$.
Since the Kobayashi metric and the Carath\'eodory metric are equal at one point,
by   the  Wong/Stanton theorem, we conclude that $X$ is biholomorphic to a ball.
  The converse is straightforward as both metrics are biholomorphically invariant with explicit formulas on a ball.

\end{proof}

 \medskip

  \begin{proof} [{\bf Proof of Corollary \ref{cor}}]
  
We will prove a negated version of the corollary, namely,
$\Omega$ is biholomorphic to a Euclidean ball if and only if the holomorphic sectional curvature of the Bergman metric is no less than  $-\frac{2}{n+1}$ and equality is attained in the sharp Lu's inequality at a point in $\Omega$.

 \medskip

The necessity holds true because for a domain 
biholomorphic to a Euclidean ball in $\mathbb{C}^n$, the 
holomorphic sectional curvature of the Bergman metric is identically equal to  $-\frac{2}{n+1}$
and  the sharp Lu's inequality \eqref{sharp} becomes equality everywhere.

\medskip

For the sufficiency, if the holomorphic sectional curvature of the Bergman metric is no less than  $-\frac{2}{n+1}$, then  \(L(\Omega) = \frac{1}{\sqrt{n+1}}\) following from the Ahlfors-Schwarz lemma. So, the holomorphic sectional curvature is no less than $-2(L(\Omega))^2$. If equality is attained in the sharp Lu's inequality at a point in $\Omega$, then  we conclude by Theorem \ref{rigidity}, Part {$\bf B$}, that $\Omega$ is biholomorphic to a  ball.

\end{proof}

\medskip

 The remainder of the paper is dedicated to understanding domains where the Bergman metric has constant negative holomorphic sectional curvature on an open set. We first establish a fundamental result about such domains through a local-to-global principle (Theorem \ref{nice}), which leads to our characterization theorem (Theorem \ref{nice cor}).

    \medskip
 
      \begin{theorem} [Local-to-global principle for Bergman metrics] \label{nice}

Let $\Omega $ be a bounded  domain in $\mathbb{C}^n$.  
Suppose the holomorphic sectional curvature of the Bergman metric $B^{2}_{\Omega}$ is identically equal to a negative constant $-c^2$ on an open set  $U\subset \Omega $. 
Then, for any $p\in U$, the Bergman representative coordinate $T (z) =(w_1, ... , w_n)^{\tau}$ 
defined by \eqref{rep}  
is a bounded holomorphic map from $ \Omega $ to the ball $$\mathcal B:=\{(w_1, ... , w_n)^{\tau} :   \sum_{\alpha, \beta=1}^n  w_\alpha   g_{ \alpha \bar \beta } (p) \overline {w_\beta }   < {2}{c^{-2}} \}.$$
Moreover, $T: (\Omega, B^{2}_{\Omega}) \to (\mathcal B, \omega_\mathcal B)$ is a holomorphic isometry,   
where  $\omega_\mathcal B$ is the Bergman metric of $\mathcal B$. 
Consequently, the holomorphic sectional curvature of the Bergman metric $B^{2}_{\Omega}$ is identically equal to a negative constant $-c^2$ on $\Omega $. 

  \end{theorem}

\begin{proof}

First, we assume that the Bergman metric at $p \in U$ 
  satisfies
    \begin{equation} \label{nor}
  g_{\alpha \bar \beta} (p) =\delta_{\alpha  \beta },
\end{equation} 
and in this case we will show that  $T (z) $ defined by \eqref{rep}  
is a bounded holomorphic map from $ \Omega $ to the ball 
$ \mathbb B_c^n := \{ w \in \mathbb C^n:   |w|^2 < {2}{c^{-2}} \}$.

 \medskip

 At $p \in U$, by the uniqueness of the Taylor expansion (see Bochner \cite{Bo47} and Lu \cite[Lemma 3]{Lu}), there exists a smaller neighbourhood $U_p$ such that the Bergman kernel can be locally decomposed as
\begin{equation}  \label{K(z, z)}
\mathcal{K}(z, z)=  \left (1 - \frac{c^2}{2} |T(z)|^2 \right )^{\frac{-2}{c^2}} e^{f(T(z))+\overline{f(T(z))}}, \quad z\in U_p,
\end{equation}
where $f$ is holomorphic on $U_p$.
The map $T$ defined by \eqref{rep} is holomorphic on $\Omega \setminus A_p$, where $A_p:=\{z\in \Omega \, | \,\mathcal{K}(z, p) =0 \, \}$ is the zero set of the Bergman kernel $\mathcal K(\cdot , p)$.
Let  $\Omega^{\prime}:= \{z \in \Omega \setminus A_p : T(z) \in \mathbb B_c^n \}$ be the set of points in $\Omega \setminus A_p$ that are mapped into the ball.  In particular, $U_p \subset \Omega^{\prime}$. By \eqref{K(z, z)} and the theory of power series, one may duplicate the variable with its conjugate so that  the full Bergman kernel can be complex analytically continued as
 $$
\mathcal{K}(z, {z_0})= \left (1 - \frac{c^2}{2} \sum_{\alpha=1}^n  w_\alpha(z)     \overline {w_\alpha(z_0)}   \right )^{\frac{-2}{c^2}}   e^{f(T(z))+\overline{f(T({z_0}))}}, \quad z, z_0\in U_p.
 $$
 
   Consider two K\"{a}hler potentials
 $$
\Phi_{p}(z):= \log \frac{\mathcal{K}(z, z)  \mathcal{K}(p, p)}{ |\mathcal{K}(z, p)|^2}, \quad  z \in \Omega \setminus A_{p}
$$
and 
$$
\Phi_{z_0}(z):= \log \frac{\mathcal{K}(z, z)  \mathcal{K}({z_0}, {z_0})}{ |\mathcal{K}(z, {z_0})|^2}, \quad  z \in \Omega \setminus A_{z_0}
$$
for the Bergman metric $B^2_\Omega=\partial \overline  \partial   \Phi_{z_0}=\partial \overline  \partial   \Phi_{p}$, for $z \in (\Omega \setminus A_{p}) \cap (\Omega \setminus A_{z_0})$. 
When $U_p$ is small enough,  for any $z_0\in U_p$, it holds that
 \begin{align*}
\Phi_{z_0}( z)&= \log \frac{\left (1 - \frac{c^2}{2} |T(z)|^2 \right )^{\frac{-2}{c^2}} e^{f(T(z))+\overline{f(T(z))}}   \left (1 - \frac{c^2}{2} |T(z_0)|^2 \right )^{\frac{-2}{c^2}} e^{f(T(z_0))+\overline{f(T(z_0))}}  }{  \left |1 - \frac{c^2}{2}   \sum_{\alpha=1}^n  w_\alpha(z)     \overline {w_\alpha(z_0)}   \right |^{\frac{-4}{c^2}} |e^{f(T(z))+\overline{f(T({z_0}))}}  |^2}\\
& =  {\frac{-2}{c^2}}   \log  \left[    \left (1 - \frac{c^2}{2} |T(z)|^2 \right )  \left (1 - \frac{c^2}{2} |T(z_0)|^2 \right )  \left  | 1 - \frac{c^2}{2}   \sum_{\alpha=1}^n  w_\alpha(z)     \overline {w_\alpha(z_0)}   \right  |^{ -2}  \right  ], \quad z\in U_p,
 \end{align*}
which yields that
$$
\Phi_{p}(z_0)= \Phi_{z_0}(p) = {\frac{-2}{c^2}} \log  \left (1 - \frac{c^2}{2} |T(z_0)|^2 \right ).
$$
Note that $\Phi_{p}(z)$ is defined on $\Omega \setminus A_p$ and thus on $\Omega^{\prime}$, where ${\frac{-2}{c^2}} \log  \left (1 - \frac{c^2}{2} |T(z )|^2 \right )$ can be defined. Since these two real-analytic  functions coincide on  $U_p$, by the identity theorem they are identical to each other on $\Omega^{\prime}$. That is, 
\begin{equation} \label{on Omega prime}
\Phi_{p}(z )=  {\frac{-2}{c^2}} \log  \left (1 - \frac{c^2}{2} |T(z )|^2 \right ), \quad z \in \Omega^{\prime}.
\end{equation} 

 \medskip
 
{\bf 1)} We claim that  no point in  $\Omega \setminus A_p$ is mapped outside the ball $\mathbb B_c^n$ by $T$. 

If not, suppose there exists some point $q \in  \Omega \setminus A_p$ that is mapped to $\{ w \in \mathbb C^n:   |w|^2 \geq {2}{c^{-2}} \}$. Choose some point $q_0 \in \Omega^{\prime}$.  Since $\Omega \setminus A_p$ is path-connected, one can choose a path $\gamma$ that connects $q_0$ and $q$. Suppose under $T$ the image of $\gamma$   intersects $\partial   \mathbb B_c^n$ firstly at some point $T(q_1)$.

 Along the path  $\gamma$ take a sequence of points $(q_l)_{l \in \mathbb N} \subset \Omega^{\prime}$ such that $q_l  \to q_1$. 
 Then by \eqref{on Omega prime},
$$
\Phi_{p}(q_l )=  {\frac{-2}{c^2}} \log  \left (1 - \frac{c^2}{2} |T(q_l)|^2 \right ).
$$
Here, as  $q_l  \to q_1$,  the left hand side is finite but  the right hand side  blows up to infinity.  This  is  a contradiction, so we have thus proved our claim, which says that $\Omega^{\prime}=  \Omega \setminus A_p $.
Therefore, \eqref{on Omega prime} in fact holds on $\Omega \setminus A_p $.

\medskip

{\bf 2)} Since $T$ maps $ \Omega \setminus A_p$ to the ball $ \mathbb B_c^n$ and satisfies 
\begin{equation} \label{bound}
|T(z)|^2<  {2}{ c^{-2}},
\end{equation}
by the Riemann removable singularity theorem, $T$ extends across the analytic variety $A_p$ to the whole domain $\Omega$ with $|T(z)|^2 \leq {2}{ c^{-2}}.$
  The maximum modulus principle yields  that \eqref{bound} in fact holds  on $\Omega$.

 \medskip

Lastly, for general $p \in U$, one performs a possible linear transformation $F$ from $\Omega$ to $\Omega_1$ such that the Bergman metric on $\Omega_1$  satisfies \eqref{nor} at $F (p)$. Since $F$ is a biholomorphism, the Bergman metric on $\Omega_1$   locally has the same holomorphic sectional curvature.

 \medskip

Denote by  $\omega_\mathcal B$ the Bergman metric on $\mathcal B$. From the above proof, we see that the K\"{a}hler potential for the Bergman metric has the following formula.
$$
\Phi_{p}( z) =  {\frac{-2}{c^2}} \log      \left (1 - \frac{c^2}{2} \sum_{\alpha, \beta=1}^n  w_\alpha(z)   g_{ \alpha \bar \beta } (p) \overline {w_\beta(z)} \right ), \quad z \in  \Omega.
$$
Direct computations then yield that $T: \Omega \to \mathcal B$ is an isometry. In fact, if the Bergman metric at $p \in U$ 
  satisfies \eqref{nor}, then the potential
$\Phi_{p}$ simplifies to
$$
\Phi_{p}( z) =  {\frac{-2}{c^2}} \log      \left (1 - \frac{c^2}{2} |T(z)|^2 \right ), \quad z \in  \Omega.
$$
So,
 \begin{align*}
g_{\alpha \bar \beta} (z)  &= \frac{\partial^2 \Phi_{p}( z) }{\partial z_\alpha \partial \overline {z_\beta} }\\
&=\sum_{i, j =1}^n  \left (1 - \frac{c^2}{2} |T(z)|^2 \right )^{-2}   \left[    \delta_{i j}  \left (1 - \frac{c^2}{2} |T(z)|^2 \right )     + \frac{c^2}{2}      \overline{w_i(z)}    w_j(z)    \right] {\frac{\partial w_i(z)}  {\partial {z_\alpha}} }  \frac{\partial  \overline{w_j (z)} }{   \partial  {\overline {z_\beta}} } \\
&= \sum_{i, j =1}^n    G_{i \bar j} (w)  {\frac{\partial w_i(z)}  {\partial {z_\alpha}} }  \frac{\partial  \overline{w_j(z)} }{   \partial  {\overline {z_\beta}} }, \quad z\in \Omega,
\end{align*}
where
$$
G_{i \bar j} (w) :=   \frac{\partial^2 }{\partial w_i \partial \overline {w_j} } \left(  \frac{-2}{c^2} \log (1 - \frac{c^2}{2} |w|^2) \right)
$$
gives  the Bergman metric  on $ \mathbb B_c^n$. The general case follows in a similar manner.   And the curvature part also follows.

\end{proof}

 \medskip

Theorem \ref{nice} directly implies a new approach to Lu's classical theorem (Theorem \ref{nice cor}), requiring constant curvature only on an open set rather than globally. This insight, combined with our study of the Lu constant, provides the key to completing the proof of Theorem \ref{2nd}.

 \medskip

  \begin{proof} [{\bf Proof of Theorem \ref{2nd}}]

From the upper bound of the holomorphic sectional curvature of the Bergman metric and the Ahlfors-Schwarz lemma (see \cite{Ahlfors, Roy, Yau78}), one obtains the first inequality in the following inequality chain
$$
K_{\Omega} \geq L(\Omega) B_\Omega  \geq C_\Omega,
$$
where the second inequality holds true due to the definition of $L(\Omega)$. 
By Theorem \ref{local}, there exists some $\epsilon > 0$ such that: if $z \in \Omega$ and $\delta_\Omega(z) < \epsilon$, then there is a nonempty, open set $E_\Omega(z)$ of directions where the Kobayashi metric $K_{\Omega}$ and the Carath\'eodory metric $C_{\Omega}$ coincide. One gets  on $E_\Omega(z)$ that  
$$
K_{\Omega} = L(\Omega)  B_\Omega = C_\Omega.
$$
By the same argument as in the proof of Theorem \ref{Main} --- namely, using the curvature bounds of the third author \cite{W} together with the fact that the holomorphic sectional curvature function $v \mapsto H(z; v)$ is algebraic (with $z$ fixed) --- we conclude that the holomorphic sectional curvature $H(z; v)$ of the Bergman metric is equal to $-2(L(\Omega))^2$ for all non-zero $v \in T_z(\Omega)$, whenever $\delta_\Omega(z) < \epsilon$.
Then, the holomorphic sectional curvature of the Bergman metric $B^2_\Omega$ is equal to $-2(L(\Omega))^2$ on $\Omega \setminus \Sigma$, where $\Sigma$ is defined by \eqref{sigma}. 
Since the Bergman metric $(\Omega, B^{2}_{\Omega})$ is complete,
Theorem \ref{nice cor} yields that $\Omega$ must be biholomorphic to $\mathbb{B}^n$.

  \end{proof}

    \medskip
      
 \subsection*{Statements and Declarations}

No financial or non-financial interests that are directly or indirectly related to the work submitted for publication were reported by the authors.

Data sharing is not applicable to this article as no datasets were generated or analysed during the current study.

\subsection*{Funding}
The research of the first author was partially supported by an AMS-Simons travel grant and the NSF grant DMS-2103608.

\subsection*{Acknowledgements}

 The authors sincerely thank \L{}ukasz Kosiński, Bernhard Lamel, Damin Wu and Andrew Zimmer for their comments. The first author thanks Song-Ying Li, Ming Xiao, and Hang Xu for stimulating discussions. We are also grateful to the anonymous referees for their careful reading and valuable suggestions, which have improved the exposition of this paper.

\fontsize{11}{9}\selectfont

\vspace{0.5cm}

\noindent xindong.math@outlook.com

\vspace{0.2 cm}

\noindent Department of Mathematics, University of Connecticut, Stamford, CT 06901-2315, USA

\vspace{0.4cm}

\noindent rwang198@ucr.edu, 

\vspace{0.2 cm}

\noindent Department of Mathematics, University of California, Riverside, CA 92521-0429, USA

\vspace{0.4cm}

\noindent wong@math.ucr.edu,

\vspace{0.2 cm}

\noindent Department of Mathematics, University of California, Riverside, CA 92521-0429, USA

\end{document}